\documentclass{amsart}

\usepackage[pdftex]{hyperref}

\hypersetup{ colorlinks=true, linkcolor=blue, filecolor=magenta, urlcolor=cyan, }

\hypersetup{
     colorlinks=true,       
    linkcolor=red,          
    citecolor=blue,        
    filecolor=cyan,         
    urlcolor=magenta        
}

\newcommand{\RR}{\mathbb{R}} 
\newcommand{\CC}{\mathbb{C}} 

\newcommand{\UB}{\mathbb{B}}

\newcommand{\HP}{\mathbb{H}}

\newcommand{\RE}{\mathrm{Re}\,}

\newcommand{\Aut}{\mathrm{Aut}}

\newcommand{\Isom}{\mathrm{Isom}}

\newcommand{\grad}{\mathrm{grad}}

\newcommand{\Ric}{\mathrm{Ric}}
\newcommand{\cK}{\mathsf{K}}

\newcommand{\dc}{d^c}

\newcommand{\vV}{\mathcal{V}}

\newcommand{\afrac}[2]{\genfrac{}{}{0pt}{}{#1}{#2}}

\newcommand{\paren}[1]{\left(#1\right)}
\newcommand{\abs}[1]{\left\lvert#1\right\rvert}
\newcommand{\norm}[1]{\left\|#1\right\|}
\newcommand{\set}[1]{\left\{#1\right\}}
\newcommand{\ip}[1]{\left\langle#1\right\rangle}
\newcommand{\pd}[2]{\frac{\partial#1}{\partial#2}}

\newcommand{\pdl}[2]{\partial#1/\partial#2}

\newcommand{\ind}[4]{{#1}^{\phantom{#2}#3}_{#2\phantom{#3}#4}}

\newcommand{\im}{\sqrt{-1}}

\newcommand{\KE}{K\"ahler-Einstein }

\newtheorem{theorem}{Theorem}[section]

\newtheorem{lemma}[theorem]{Lemma}
\newtheorem{proposition}[theorem]{Proposition}
\theoremstyle{definition}

\theoremstyle{remark}
\newtheorem{remark}[theorem]{Remark}

\numberwithin{equation}{section}

\title[A characterization of the unit ball]{A characterization of the unit ball by a K\"ahler-Einstein potential}

\author{Young-Jun Choi}
\address{Department of Mathematics, 
Pusan National University,
2, Busandaehak-ro 63beon-gil,
Geumjeong-gu, Busan, 46241, Republic of Korea}
\email{youngjun.choi@pusan.ac.kr}

\author{Kang-Hyurk Lee}
\address{Department of Mathematics and Research Institute of Natural Science, 
Gyeongsang National University, 
Jinju, Gyeongnam, 52828, Republic of Korea}
\email{nyawoo@gnu.ac.kr}

\author{Aeryeong Seo}
\address{Department of Mathematics and RIRCM,
Kyungpook National University,
80, Daehak-ro, Buk-gu, Daegu, 41566, Republic of Korea}
\email{aeryeong.seo@knu.ac.kr}



\thanks{This work was supported by Samsung Science and Technology Foundation under Project Number SSTF-BA2201-01. The first named author was supported by the National Research Foundation of Korea (NRF) grant funded by the Korea government (No. 2018R1C1B3005963, No. 2021R1A4A10324181262182065300102).
The third  named author was supported by the National Research Foundation of Korea (NRF) grant funded by the Korea government (No. NRF-2022R1F1A1063038) }

\dedicatory{Dedicated to Professor Kang-Tae Kim on the occasion of his 65th birthday}

\begin{document}

\begin{abstract}
We will show that a universal covering of a compact K\"ahler manifold with ample canonical bundle is the unit ball if it admits a global potential function of the K\"ahler-Einstein metric whose gradient length is a minimal constant. As an application, we will extend the Wong-Rosay theorem to a complex manifold without boundary.
\end{abstract}


\maketitle

\section{Introduction}

Most, but not all, negatively curved compact K\"ahler manifolds are covered by bounded symmetric domains. Thus it is natural to distinguish bounded symmetric domains from exceptional spaces, such as the universal covering of the Mostow-Siu surface (\cite{Mostow-Siu1980}). From this point of view, bounded symmetric domains have been characterized as a bounded domain with relevant conditions to the Cartan/Harish-Chandra realization of hermitian symmetric space of noncompact type (\cite{Wong1977,Rosay1979,Frankel1989}), and as an irreducible complex manifold with nontrivial, holomorphic transformation group (\cite{Nadel1990,Frankel1995}). An important aspect in these studies is the fact that the Bergman metric of a bounded symmetric domain is the unique, biholomorphically invariant K\"ahler metric and its Ricci curvature is negative; thus a canonical bundle of any compact quotient is positive so ample. For the uniformization of a compact complex manifold with ample canonical bundle, it makes more sense to regard the Bergman metric as a complete \KE metric with negative Ricci curvature. In this paper, we shall characterize the unit ball (the bounded symmetric domain of rank $1$) by an existence of a certain potential function of the \KE metric.

As pointed out by Kai-Ohsawa~\cite{Kai-Ohsawa2007} (see also Theorem~\ref{thm:Kai-Ohsawa}), a bounded homogeneous domain (so symmetric one also) admits a global potential function of its Bergman metric whose gradient length is constant, and a value of constant gradient length is indeed uniquely determined. We shall show that  the unit ball  has the minimal value for constant gradient length among bounded symmetric domains of the same dimension (Lemma~\ref{lem:gradient length of BSDs}). Moreover we can characterize the unit ball by the same minimal value condition among universal coverings of compact K\"ahler manifolds with ample canonical bundle.

\begin{theorem}\label{thm:main1}
Let $X^n$ be a simply connected complex manifold of dimension $n$ which covers a compact complex manifold and admits a complete \KE metric $\omega$ with negative Ricci curvature $-\cK$. Suppose that there is a global potential function $\varphi:X\to\RR$ of $\omega$ satisfying
\begin{equation*}
\norm{d\varphi}_\omega^2\equiv  \frac{2(n+1)}{\cK} \;.
\end{equation*}
Then $X$ is biholomorphic to the unit ball $\UB^n=\set{z\in\CC^n:\norm{z}<1}$.
\end{theorem}

By S.~T.~Yau~\cite{Yau1978}, a compact K\"ahler manifold with ample canonical bundle has a \KE metric with negative Ricci curvature, so the lifted metric structure to its covering is complete.  On the other way, Yau's Schwarz Lemma (\cite{Yau1978-1}) implies that a compact quotient of a complete \KE manifold with negative Ricci curvature is a Riemannian quotient, so its canonical bundle is positive. Thus Theorem~\ref{thm:main1} characterizes the unit ball in the class of universal coverings of compact K\"ahler manifolds with ample canonical bundle. Usually, $-1$ or $-(n+1)$ have been used as a normalized condition to the Ricci curvature for the uniqueness and the biholomorphical invariance, but we shall take an arbitrary negative constant $-\cK$ because the minimal value $2(n+1)/\cK$ depends on $\cK$ and we would like to indicate where the Ricci curvature is involved.

Due to the Nadel-Frankel theorem (\cite{Nadel1990,Frankel1995}), the manifold $X$ in Theorem~\ref{thm:main1} is a product space of a bounded symmetric domain and a complex manifold of discrete automorphism group (a rigid factor) since it covers a compact K\"ahler manifold with ample canonical bundle. On the other hand, the condition to the gradient length of the potential function $\varphi$ allows us to construct a complete holomorphic vector field from the gradient vector field $\grad(\varphi)$ (Theorem~3.2 in \cite{Choi-Lee2021}, see also Theorem~\ref{thm:existence} of this manuscript). We will mainly show in the proof (Section~\ref{subsec:proof of thm1}) that the specific construction of $V$ forces the rigid factor of $X$ to be trivial, so $X$ is equivalent to a bounded symmetric domain. Lemma~\ref{lem:gradient length of BSDs}, a characterization of the unit ball among bounded symmetric domains, implies that the unit ball is the only candidate to be biholomorphic to $X$.

\medskip

Theorem~\ref{thm:main1} gives an intrinsic generalization of the Wong-Rosay theorem (\cite{Wong1977,Rosay1979}) which says that \textit{a smoothly bounded domain in $\CC^n$ which admits a compact quotient is biholomorphically equivalent to the unit ball}.

Let $\Omega$ be  a bounded domain  in $\CC^n$ with $C^2$-smooth boundary. If the holomorphic automorphism group $\Aut(\Omega)$ of $\Omega$ admits  a discrete,  torsion-free, cocompact subgroup  $\Gamma$, that is, the quotient complex space $\Gamma\backslash\Omega$ is compact, then there is a sequence $\set{f_j}\subset\Gamma$  whose orbit $\set{f_j(p)}$ for $p\in\Omega$ accumulates at a strongly pseudoconvex boundary point $q\in\partial\Omega$. In \cite{Wong1977,Rosay1979},  it is proved that the asymptotic value of  the ratio of Eisenman-Kobayashi and Carath\'eodory measures at a strongly pseudoconvex boundary point is $1$. Since the ratio is invariant under automorphisms, one can see the ratio at $p$ is $1$ so the domain $\Omega$ is biholomorphic to the unit ball.  From an improvement by Efimov~\cite{Efimov1995}, we can directly construct a biholomorphism to the unit ball using the affine rescaling method even in case of  a complex manifold with boundary (see \cite{Gaussier-Kim-Krantz2002}).

Here, we will try to extend the Wong-Rosay theorem to a complex manifold without boundary. Let us go back to the domain $\Omega\subset\CC^n$ and its cocompact subgroup $\Gamma\subset\Aut(\Omega)$ as above. Admitting a compact quotient $\Gamma\backslash\Omega$ implies that $\Omega$ is pseudoconvex so has a complete \KE metric with Ricci curvature $-\cK$ by Cheng-Yau~\cite{Cheng-Yau1980} and Mok-Yau~\cite{Mok-Yau1983}. Then the gradient length $\norm{d\varphi}_\omega^2$ of the canonical potential function $\varphi$ of $\omega$ is uniformly bounded near  a strongly  pseudoconvex boundary point $q$ and the asymptotic value at $q$ is $2(n+1)/\cK$:
\begin{equation*}
\lim_{\afrac{z\to q}{z\in\Omega}} \norm{d\varphi}_\omega^2(z) =\frac{2(n+1)}{\cK}
\end{equation*}
(see Proposition~\ref{prop:boundary behavior}). Using the method of potential rescaling as in \cite{LKH2021}, we can construct a global potential function $\tilde\varphi$ of $\omega$ whose gradient length is the minimal constant:
\begin{equation*}
\norm{d\tilde\varphi}_\omega^2\equiv \frac{2(n+1)}{\cK} \;.
\end{equation*}
Since $\Omega$ is simply connected (\cite{Wong1977}), Theorem~\ref{thm:main1} implies that $\Omega$ is the unit ball up to biholomorphic equivalence. More generally, we have
\begin{theorem}\label{thm:main2}
Let $X^n$ be a complex manifold equipped with the complete K\"ahler-Einstein metric $\omega$ with negative Ricci curvature $-\cK$ and suppose that there is  a discrete,  torsion-free, cocompact subgroup  $\Gamma$  of $\Aut(X)$. If
\begin{enumerate}
\item there is a sequence $\set{f_j}\subset\Gamma$ with a localizing neighborhood $U$;
\item there is a local potential function  $\varphi:U\to\RR$ of $\omega$ which has uniformly bounded  gradient length and satisfies 
\begin{equation}\label{cond:boundary behavior}
\lim_{j\to\infty} \norm{d\varphi}_\omega^2 (f_j(x)) = \frac{2(n+1)}{\cK} \quad\text{for any $x\in X$,}
\end{equation}
\end{enumerate}
then $X$ is a quotient space of the unit ball $\UB^n$. If $U$ is simply connected in addition, then $X$ is biholomorphic to the unit ball. 
\end{theorem}

Here, a \emph{localizing neighborhood} $U$ of $\set{f_j}$ means that  $f_j(K)\subset U$ for any compact subset $K\subset X$ with sufficiently large $j$. Thus the value $\norm{d\varphi}_\omega^2 (f_j(x))$ in \eqref{cond:boundary behavior} is defined well. In case of a domain $\Omega\subset\CC^n$ with a cocompact subgroup $\Gamma$, an intersection $\Omega\cap U$ for any open neighborhood $U$ of a strongly pseudoconvex boundary point $q\in\partial\Omega$ should be a localizing neighborhood of a sequence $\set{f_j}\subset\Aut(\Omega)$ whose orbit accumulates at $q$. As we mentioned, Condition~\eqref{cond:boundary behavior} also holds for $\set{f_j}$.

In the proof of Theorem~\ref{thm:main2} (Section~\ref{subsec:proof of thm2}), we will apply the method of potential rescaling in \cite{LKH2021,Choi-Lee2021} to construct a global potential function with constant gradient length $2(n+1)/\cK$. Then Theorem~\ref{thm:main1} allows us to conclude that the universal covering space is equivalent to the unit ball.

\medskip

The organization of this paper is as follows. We will introduce the gradient length of potentials, some relevant identities and basic materials for the main results in Section~\ref{sec:section2}; especially the characterization of the unit ball among bounded symmetric domains by the minimal length constant $2(n+1)/\cK$ (Section~\ref{subsec:minimal length of gradient}) and the existence of a complete holomorphic vector field in case of the minimal length constant (Section~\ref{subsec:existence of a complete holomorphic vector field}, see also \cite{Choi-Lee2021}). Then we will prove Theorem~\ref{thm:main1} and Theorem~\ref{thm:main2} in Section~\ref{sec:proof of main results}.



\section{A \KE potential with constant gradient length and the existence of a complete holomorphic vector field}\label{sec:section2}
We will discuss now a negatively curved, complete \KE manifold admitting a global potential function with constant gradient length. In Lemma~\ref{lem:gradient length of BSDs}, we will characterize the unit ball among bounded symmetric domains by a minimal constant of gradient length.  Then the existence theorem of a complete holomorphic vector field in case of the minimal constant will be introduced.

\medskip

Throughout this section, $X^n$ is an $n$-dimensional complex manifold  and $\omega$ is its complete \KE metric with negative Ricci curvature $-\cK$, that is,
\begin{equation*}
\Ric(\omega)=-\cK\omega\quad 
\end{equation*}
where $\omega$ also stands for its K\"ahler form. We will employ $z=(z^1,\ldots,z^n)$ as a local coordinate system for local expressions of quantities. In this local coordinates, $(g_{\alpha\bar\beta})$ stands for the metric tensor of $\omega$:
\begin{equation*}
\omega = \im g_{\alpha\bar\beta} dz^\alpha\wedge dz^{\bar\beta}
\end{equation*}
where the indices $\alpha,\beta,\ldots$ run from $1$ to $n$ and the summation convention for duplicated indices is always assumed. We denote the complex conjugate of a tensor by taking the bar on the indices: $\overline{z^\alpha} = z^{\bar\alpha}$, $\overline{g_{\alpha\bar\beta}} = g_{\bar\alpha\beta}$. We will also use the matrix $(g_{\alpha\bar\beta})$ and  its inverse matrix $(g^{\bar\beta\alpha})$ to raise and lower indices: 
\begin{equation*}
\varphi^\alpha = g^{\alpha\bar\beta}\varphi_{\bar\beta}\;, \quad \ind{R}{\beta}{\alpha}{\mu\bar\nu}=g^{\bar\gamma\alpha}R_{\beta\bar\gamma\mu\bar\nu} \;.
\end{equation*}


\subsection{Gradient length of a potential function}
A \emph{potential function} $\varphi$ of $\omega$ is a local smooth function satisfying
\begin{equation*}
d\dc\varphi = \omega \;.
\end{equation*}
Here $\dc=\frac{\im}{2}(\bar\partial-\partial)$ so $d\dc=\im\partial\bar\partial$. Locally, a potential function is naturally given by
\begin{equation}\label{eqn:canonical potential}
\varphi=\frac{1}{\cK}\log\det(g_{\alpha\bar\beta})
\end{equation}
since $\omega$ is \KE and the Ricci form is  written by $$\Ric(w)=-d\dc\log\det(g_{\alpha\bar\beta}).$$ For a bounded pseudoconvex domain $\Omega$ in $\CC^n$ and its \KE metric $\omega$, the function given by \eqref{eqn:canonical potential} with respect to the standard coordinates of $\CC^n$ is a global potential function of $\omega$, so called a \emph{canonical potential} of $(\Omega,\omega)$.

The \emph{gradient length} of a potential function $\varphi$ is the length of the $1$-form $d\varphi$ measured by $\omega$ which is locally expressed by
\begin{equation*}
\norm{d\varphi}_\omega^2 = \norm{\varphi_\alpha dz^\alpha+\varphi_{\bar\beta} dz^{\bar\beta}}_\omega^2= 2\varphi_\alpha\varphi_{\bar\beta}g^{\alpha\bar\beta} = 2\varphi_\alpha\varphi^\alpha
\end{equation*}
where $\varphi_\alpha=\pdl{\varphi}{z^\alpha}$, $\varphi_{\bar\beta}=\pdl{\varphi}{z^{\bar\beta}}$. For the sake of simplicity, we will often employ 
\begin{equation*}
\norm{\partial\varphi}_\omega^2 = \varphi_\alpha\varphi^\alpha =\frac{1}{2} \norm{d\varphi}_\omega^2
\end{equation*}
as a gradient length.

\begin{remark}\label{rmk:change of Ricci}
If $\omega'$ is another complete \KE metric with Ricci curvature $-\cK'$, then by the uniqueness of \KE metric (\cite{Yau1978-1}), we have
\begin{equation*}
\cK\omega = \cK'\omega'
\end{equation*}
since the Ricci curvature tensor is invariant under the scalar multiplication to the metric so
\begin{equation*}
-\cK\omega= \Ric(\omega)=\Ric(\omega')=-\cK'\omega' \;.
\end{equation*}
Given potential $\varphi'$ of $\omega'$, a corresponding potential of $\omega$ is of the form $\varphi=\cK'\varphi'/\cK$, so we have a relation of their gradient length:
\begin{equation}\label{eqn:change of Ricci}
\norm{\partial\varphi}_\omega^2 = \paren{\frac{\cK'}{\cK}}^2\norm{\partial\varphi'}_\omega^2 = \frac{\cK'}{\cK}\norm{\partial\varphi'}_{\omega'}^2 \;.
\end{equation}
\end{remark}


\subsection{Boundary behavior of gradient length}\label{subsec:boundary behavior of gradient length}
Let us consider the \KE metric of the unit ball $\UB^n=\set{z\in \CC^n:\norm{z}<1}$. The defining function $\rho(z)=\norm{z}^2-1$ gives a potential function
\begin{equation}\label{eqn:canonical potential of UB}
\varphi_\rho
	=-\log(-\rho)
	=-\log\paren{1-\norm{z}^2}
	=\frac{1}{n+1}\log\frac{1}{\paren{1-\norm{z}^2}^{n+1}}
\;,
\end{equation}
of the complete \KE metric $\omega_{\UB^n}$ with Ricci curvature $-(n+1)$ (i.e. $\cK=n+1$), whose metric tensor is given by
\begin{equation*}
g_{\alpha\bar\beta}^{\UB^n}
	=\frac{n+1}{\paren{1-\norm{z}^2}^2}
	\paren{
		\delta_{\alpha\bar\beta}(1-\norm{z}^2)
		+z^{\bar\alpha}z^\beta
	} \;.
\end{equation*}
One can easily see that $\varphi_\rho$ is the canonical potential function of $\omega^{\UB^n}$. From
\begin{equation*}
g^{\alpha\bar\beta}_{\UB^n}
	=\frac{(1-\norm{z}^2)}{n+1}
	\paren{
		\delta^{\alpha\bar\beta}-z^\alpha z^{\bar\beta}
	}
	\;,
\end{equation*}
we have
\begin{equation*}
\norm{\partial\varphi_\rho}_{\omega}^2
	=\norm{z}^2
	\;.
\end{equation*}
This implies that the boundary value of $\norm{\partial\varphi_\rho}_{\omega}^2$ is $1$. For the \KE metric $\omega=\frac{n+1}{\cK}\omega_{\UB^n}$ with Ricci curvature $-\cK$ and its potential function $\varphi= \frac{n+1}{\cK}\varphi_\rho$, one can see that 
\begin{equation*}
\norm{\partial\varphi}_\omega^2(z) \to \frac{n+1}{\cK} \;.
\end{equation*}
as $z$ tends to $\partial\UB^n$ from \eqref{eqn:change of Ricci}.

In Proposition~4.3 in \cite{Choi-Lee2021}, we proved that such boundary behavior of gradient length still holds for strongly pseudoconvex bounded domains. From \cite{Gontard2019}, we can also have the same estimate near a strongly pseudoconvex boundary point.
\begin{proposition}\label{prop:boundary behavior}
Let $\Omega\subset\CC^n$ be a bounded pseudoconvex domain with $C^k$ boundary where $k\geq \max\set{2n+9,3n+6}$ and let $\varphi$ be the canonical potential of the complete \KE metric $\omega$ with Ricci curvature $-\cK$. If $q\in\partial\Omega$ is a strongly pseudoconvex boundary point, then
\begin{equation*}
\norm{\partial\varphi}_\omega^2(z) \to\frac{n+1}{\cK}
\end{equation*}
as $z\to q$. 
\end{proposition}

\begin{proof}
This is proved in \cite{Choi-Lee2021} when $\Omega=\{z\in\CC^n:r(z)<0\}$ is a bounded strongly pseudoconvex domain with a smooth boundary by using the boundary behavior of the solution of the following complex Monge-Ampere equation \cite{Cheng-Yau1980}:
\begin{equation}\label{E:CMAE}
\begin{aligned}
\paren{\omega+dd^cu}^n&=e^{\cK u+F}\omega^n,
\\
\omega+dd^cu&>0,
\end{aligned}
\end{equation}
where $\omega=-\log(-r)$ and $F=\log\det(r_{\alpha\bar\beta})(-r+\abs{dr}^2)$.
More precisely, if $r$ is a defining function, which is called an approximate solution of the complex Monge-Ampere equation, satisfying
\begin{equation*}
F
=
\log\det(r_{\alpha\bar\beta})(-r+\abs{\partial r}^2)
=
O(\abs{r}^{n+1}),
\end{equation*}
then $u$ satisfies that
\begin{equation}\label{E:boundry_behavior}
\abs{D^pu}(x)=O(\abs r^{n+1/2-p-\varepsilon})
\quad\text{for}\quad \varepsilon>0.
\end{equation}
Since $\varphi=-\log(-r)+u$, one can easily compute the boundary behavior of $\norm{\partial\varphi}^2$.

In case that $\Omega$ is a bounded pseudoconvex domain with the hypothesis, Gontard proved that there exists a local approximate solution $r$ in a neighborhood $U\subset\CC^n$ of $q\in\partial\Omega$ (\cite{Gontard2019}). He also proved that the local solution $u$ of \eqref{E:CMAE} in $U\cap\Omega$ with a local approximate solution $r$ satisfies the boundary behavior \eqref{E:boundry_behavior}. Then the same computation as in the previous case gives the conclusion. (For the detailed proof, see Section 4 in \cite{Choi-Lee2021}.)
\end{proof}

\subsection{The minimal value of constant gradient length}\label{subsec:minimal length of gradient}
For a bounded symmetric domain $\Omega$ in $\CC^n$, the Bergman metric $\omega_\Omega=d\dc\log K_\Omega$ given by the Bergman kernel function $K_\Omega$ is the complete \KE metric. As an extended study of \cite{Donnelly-Fefferman1983} for the $L^2$-cohomology vanishing, Donnelly~\cite{Donnelly1997} proved that the potential function $\log K_\Omega$  has bounded gradient length in order to get a kind of K\"ahler-hyperbolicity of the Bergman metric (\cite{Gromov1991}). In the paper of Kai-Ohsawa (\cite{Kai-Ohsawa2007}), they considered the Cayley transform to a Siegel domain of the second kind. Since a homogeneous Siegel domain of the second kind is affine-homogeneous, its Bergman kernel function generates a potential function with constant gradient length. Moreover a value of constant gradient length does not depend on any choice of potentials (see also Theorem~\ref{thm:Kai-Ohsawa}). Under the normalized condition $-\cK$ to the Ricci curvature, we will prove that \emph{the gradient length constant of the unit ball is minimal among bounded symmetric domains.}

\medskip

Let us denote by $\Delta_\omega$ the Laplace-Beltrami operator  of $\omega$ with non-positive eigenvalues which is locally written by
\begin{equation*}
\Delta_\omega = g^{\alpha\bar\beta}\nabla_\alpha\nabla_{\bar\beta}
\end{equation*}
where $\nabla$ is the covariant derivative of $\omega$  and $\nabla_\alpha=\nabla_{\pdl{}{z^\alpha}}$, $\nabla_{\bar\beta}=\nabla_{\pdl{}{z^{\bar\beta}}}$. 

\begin{proposition}[Proposition~3.1 in \cite{Choi-Lee2021}]\label{P:PDE}
Let $\varphi$ be a local potential function of a complete \KE manifold $(X^n,\omega)$ with negative Ricci curvature $-\cK$. Then  
\begin{equation*}
	\Delta_\omega \norm{\partial\varphi}^2_\omega
	=
	\norm{\nabla'^2 \varphi}_\omega^2
	+
	n-\cK\norm{\partial\varphi}^2_\omega \;.
\end{equation*}
\end{proposition}
Here, $\nabla'$ is the $(1,0)$-part of $\nabla$. The length $\norm{\nabla'^2 \varphi}_\omega^2$   can be locally written by
\begin{multline*}
\norm{\nabla'^2 \varphi}_\omega^2 
= \norm{(\nabla_\beta\nabla_\alpha\varphi )dz^\alpha\otimes dz^\beta}_\omega^2 
= \norm{\varphi_{\alpha;\beta}dz^\alpha\otimes dz^\beta}_\omega^2 
\\
= \varphi_{\alpha;\beta}\varphi_{\bar\lambda;\bar\mu}g^{\alpha\bar\lambda}g^{\beta\bar\mu}
= \varphi_{\alpha;\beta} \varphi^{\alpha;\beta} \;
\end{multline*}
and coincides with the trace of the semi-positive symmetric operator
\begin{equation}\label{eqn:symmetric operator}
\ind{\varphi}{}{\alpha}{;\bar\beta}\ind{\varphi}{}{\bar\beta}{;\gamma} \pd{}{z^\alpha}\otimes dz^\gamma 
\end{equation}
of $T^{(1,0)}X$.

\medskip

Suppose that  $\norm{\partial\varphi}_\omega^2$ is locally constant. Then we have
\begin{multline*}
	0
	=
	\partial\paren{\norm{\partial\varphi}_\omega^2}
	=(\varphi_\alpha \varphi^\alpha)_{;\beta}dz^\beta
	=\paren{
		\varphi_{\alpha;\beta}\varphi^\alpha
		+\varphi_\alpha\ind{\varphi}{}{\alpha}{;\beta}
		}
		dz^\beta
	\\
	=\paren{
		\varphi_{\alpha;\beta}\varphi^\alpha
		+\varphi_\alpha\ind{\delta}{}{\alpha}{\beta}
		}dz^\beta
	=
	\paren{
		\varphi_{\alpha;\beta}\varphi^\alpha
		+
		\varphi_\beta
	}dz^\beta.
\end{multline*}
It follows that
\begin{equation}\label{E:Key_equation}
\varphi_{\alpha;\beta}\varphi^\alpha = \varphi_{\beta;\alpha}\varphi^\alpha
=
-\varphi_\beta.
\end{equation}
This means that at each point of $X$ where $\varphi$ is defined, the gradient vector
\begin{equation*}
\mathrm{grad}(\varphi) = \varphi^\alpha\pd{}{z^\alpha} = g^{\alpha\bar\beta}\varphi_{\bar\beta}\pd{}{z^\alpha}
\end{equation*}
is an eigenvector of the semi-positive symmetric operator in \eqref{eqn:symmetric operator} with the eigenvalue $1$. Therefore $\norm{\nabla'^2 \varphi}_\omega^2\geq 1$. As a conclusion, we have
\begin{equation*}
0=\Delta_\omega \norm{\partial\varphi}_\omega^2
	=
	\norm{\nabla'^2 \varphi}_\omega^2
	+
	n
	-\cK\norm{\partial\varphi}_\omega^2
	\geq
	1+n-\cK\norm{\partial\varphi}_\omega^2
	\;.
\end{equation*}
This implies the following.
\begin{proposition}\label{prop:lower bound of length}
Let $\varphi$ be a local potential function of the \KE metric $\omega$ with negative Ricci curvature $-K$. 
If the length $\norm{\partial\varphi}_\omega$  is constant, then
\begin{equation*}
\norm{\partial\varphi}_\omega^2 \geq \frac{n+1}{\cK} \;.
\end{equation*}
\end{proposition}

Now we will see that only the unit ball has a global \KE potential function whose gradient length attains the optimal (so minimal) constant $(n+1)/\cK$.


\subsubsection{A potential of the unit ball $\UB^n$} 
Let $\omega_{\UB^n}$ be the \KE metric of $\UB^n$ with Ricci curvature $-(n+1)$ and $\varphi_\rho$ be its canonical potential as in Section~\ref{subsec:boundary behavior of gradient length}. One can construct a potential function 
\begin{equation*}
\tilde\varphi = \varphi_\rho + 2\log\abs{1+z^1}
	=\frac{1}{n+1}\log\frac{\abs{1+z^1}^{2(n+1)}}{\paren{1-\norm{z}^2}^{n+1}}
\;,
\end{equation*}
applying the method of potential rescaling in \cite{LKH2021} to a sequence of hyperbolic automorphisms whose orbit accumulates at $(1,0,\ldots,0)\in\partial\UB^n$. By Proposition~2.2 in \cite{LKH2021}, this $\tilde\varphi$ has a constant gradient length
\begin{equation*}
\norm{\tilde\varphi}_{\omega_{\UB^n}}^2 \equiv 1  \;.
\end{equation*}
Therefore
\begin{equation*}
\varphi = \frac{n+1}{\cK}\tilde\varphi =\frac{1}{\cK}\log\frac{\abs{1+z^1}^{2(n+1)}}{\paren{1-\norm{z}^2}^{n+1}}
\end{equation*}
is a potential function of the complete \KE metric $\omega$ with Ricci curvature $-\cK$ satisfying
\begin{equation*}
\norm{\partial\varphi}_\omega^2 \equiv \frac{n+1}{\cK} \;.
\end{equation*}


\subsubsection{A potential for the bounded symmetric domain}

Irreducible bounded symmetric domains consist of the following 
four classical type domains,
\begin{align*}
\Omega_{p,q}^{\mathrm{I}} &= \left\{ Z\in M^{\mathbb C}(p,q) : I_p -  ZZ^* >0 \right\} \;,\\
\Omega^{\mathrm{II}}_m &= \left\{ Z\in  M^{\mathbb C}(m,m) : I_m - ZZ^*>0, \,\, Z^t = -Z  \right\}\;,\\
\Omega_{m}^{\mathrm{III}} &= \left\{ Z\in M^{\mathbb C}(m,m) : I_m - ZZ^*  >0,\,\, Z^t = Z  \right\}\;,\\
\Omega_m^{\mathrm{IV}} &= \left\{ Z=(z_1,\ldots, z_m)\in {\mathbb C}^m : 
ZZ^* <1 ,\, 0< 1-2 ZZ^*+ \left| ZZ^t \right|^2   \right\}\;,
\end{align*}
and two exceptional type domains,
\begin{equation*} 
\Omega_{16}^{\mathrm{V}}\;, \quad \Omega_{27}^{\mathrm{VI}}\;.
\end{equation*}
Here $M^{\mathbb C}(p,q)$ denotes the set of $p\times q$ complex matrices and $Z^*$ the complex conjugate transpose of the matrix $Z\in M^{\mathbb C}(p,q)$.

Let $\Omega$ be an irreducible bounded symmetric domain in $\CC^n$. The Bergman kernel $K_\Omega$ is of the form 
\begin{equation*}
K_\Omega(z,z) = c N_\Omega(z,z)^{-c_\Omega}
\end{equation*}
for the generic norm $N_\Omega$ of $\Omega$ and some positive constants $c$, $c_\Omega$. The constant $c_\Omega$, the dimension $n$ and the rank $r$ are given by as follows.

\setlength{\tabcolsep}{10pt}
\renewcommand{\arraystretch}{1.5}
\begin{table}[h]
\label{data for BSDs}
\caption{Invariants}
\begin{tabular}{ c|c|c|c|c|c|c}
$\Omega$& $\Omega_{p,q}^{\mathrm{I}}$&$\Omega_{m,m}^{\mathrm{II}}$& $\Omega^{\mathrm{III}}_{m,m}$&$\Omega_m^{\mathrm{IV}}$  &$\Omega_{16}^{\mathrm{V}}$& $\Omega_{27}^{\mathrm{VI}}$\\\hline
$c_\Omega$&$p+q$&$2(m-1)$&$m+1$&$m$&$12$&$18$\\\hline
$n$&$pq$&$\frac{m(m-1)}{2}$&$\frac{m(m+1)}{2}$&$m$&$16$&$27$\\\hline
$r$&$p$&$\left[ \frac{m}{2}\right]$&$m$&$2$&$2$&$3$
\end{tabular}
\end{table}

By abuse of notation, we will denote the Bergman kernel function by $K_\Omega$, that is, $K_\Omega(z)=K_\Omega(z,z)$.
Let us consider the Bergman metric $\omega_\Omega$ of $\Omega$, 
\begin{equation*}
\omega_\Omega=d\dc \log K_\Omega \;,
\end{equation*}
which is also a complete \KE metric with Ricci curvature $-1$ (i.e. $\cK=1$). We remark that the potential function $\log K_\Omega$ does not have a constant gradient length. 

Let us consider a Cayley transform $\sigma:\Omega\to S$ where $S$ is a Siegel domain of the second kind. Then the Bergman kernel function $K_S$ of $S$ gives a potential function $\sigma^*\log K_S$ of $(\Omega,\omega_\Omega)$ with constant gradient length.
\begin{theorem}[Kai-Ohsawa~\cite{Kai-Ohsawa2007}]\label{thm:Kai-Ohsawa}
The potential function $\sigma^*\log K_S$ of $\omega_\Omega$ has a constant gradient length $L_\Omega$:
\begin{equation*}
\norm{\partial(\sigma^*\log K_S)}_{\omega_\Omega}^2 \equiv L_\Omega \;.
\end{equation*}
If there is another potential function $\varphi$ with constant $\norm{\partial\varphi}_{\omega_\Omega}^2$, then $\norm{\partial\varphi}_{\omega_\Omega}^2\equiv L_\Omega$.
\end{theorem}

For a maximal totally geodesic polydisc $\Delta^r$ in $\Omega$ where $r$ is the rank of $\Omega$, we may assume that 
\begin{equation*}
\Delta^r= \set{(z^1,\ldots, z^r, 0\ldots,0) : \abs{z^\alpha}<1\text{ for $\alpha=1,\ldots,r$}}
\end{equation*}
by a change of coordinates of $\CC^n$. Then we can take a Cayley transform $\sigma\colon \Omega\to S$
such that
\begin{enumerate}
\item $S$ is a Siegel domain of the second kind in $\CC^n$,
\item the restriction $\sigma|_{\Delta^r} \colon\Delta^r\to \mathbb H^r\subset S$ is a Cayley transformation of the polydisc to the $r$-product of the right half plane $\HP=\set{\zeta\in\CC:\RE \zeta < 0}$,
\begin{equation*}
\mathbb H^r=\{(w^1,\ldots, w^r,0,\ldots,0): \RE w^\alpha <0\text{ for $\alpha=1,\ldots,r$}\} \;,
\end{equation*}
more precisely
\begin{equation}\label{eqn:Cayley transformation on polydisc}
\sigma(z^1,\ldots,z^r,0,\ldots,0) = \paren{\frac{z^1-1}{z^1+1},\ldots,\frac{z^r-1}{z^r+1},0,\ldots,0} \;.
\end{equation}
\end{enumerate}
Note that at $w=(w^1,\ldots,w^r,0,\ldots,0)\in\HP^r$, the Bergman kernel $K_S$ of $S$ is written by
\begin{equation}\label{eqn:Bergam kernel on HP}
K_S(w,w) = C \paren{K_{\HP^r }(w,w)}^{c_\Omega/2} = C\paren{K_{\HP}(w^1,w^1)\cdots K_{\HP}(w^r,w^r)}^{c_\Omega/2}
\end{equation}
for some positive constant $C$ where $K_\HP$ is the Bergman kernel of $\HP$.

By a straightforward calculation, the Bergman metric  is $c_\Omega \delta_{\alpha\bar\beta}$ at $0\in\Omega$, so we get
\begin{multline*}
L_\Omega=\norm{\partial(\sigma^*\log K_S)}_{\omega_\Omega}^2(0)
	=
\norm{\partial \log (K_{S}\circ\sigma)}_{\omega_\Omega}^2(0)
\\
	= \norm{
		\sum_{\alpha=1}^n 
		\frac{\partial}{\partial z^\alpha} \log (K_{S}\circ\sigma)dz^\alpha
		}_{\omega_\Omega}^2(0)
	=\frac{1}{{c_\Omega}}\sum_{\alpha=1}^n \left| \left.\frac{\partial}{\partial z^\alpha}\right|_{z=0} \log(K_{S}\circ\sigma) \right|^2
	\\
	\geq \frac{1}{{c_\Omega}}\sum_{\alpha=1}^r \left| \left.\frac{\partial}{\partial z^\alpha}\right|_{z=0} \log (K_{S}\circ\sigma) \right|^2 \;.
\end{multline*}
Since $\sigma(0)=(-1,\ldots,-1,0,\ldots,0)\in\HP^r$ and $K_{\HP}(\zeta) = 2/\paren{\zeta+\bar\zeta}^2$, Equation~\eqref{eqn:Cayley transformation on polydisc} and the identity in \eqref{eqn:Bergam kernel on HP} implies that
\begin{equation*}
\left.\frac{\partial}{\partial z^\alpha}\right|_{z=0} \log (K_{S}\circ\sigma)
=c_\Omega \paren{\left.\pd{}{\zeta}\right|_{\zeta=-1}\log K_{\HP} }
=c_\Omega \;.
\end{equation*}
Therefore we have
\begin{equation*}
L_\Omega \geq rc_\Omega \;.
\end{equation*}
This gives a characterization of the unit ball among bounded symmetric domains.
\begin{lemma}\label{lem:gradient length of BSDs}
Let $\Omega$ be a bounded symmetric domain in $\CC^n$ with the complete \KE metric $\omega$ with Ricci curvature $-\cK$ and  let $\varphi$ be a potential function of $\omega$  with constant $\norm{\partial\varphi}_{\omega}^2$. Then if $\Omega$ is not the unit ball up to biholomorphic equivalence, then
\begin{equation*}
\norm{\partial\varphi}_{\omega}^2> \frac{n+1}{\cK} \;.
\end{equation*} 
\end{lemma}

\begin{proof}
Since the Ricci curvature of the Bergman metric $\omega_\Omega$ is $-1$, so $\omega=(1/\cK)\omega_\Omega$ is the complete \KE metric with Ricci curvature $-\cK$ and $\varphi=(1/\cK)\sigma^*\log K_S$ is a potential function of $\omega$ satisfying
\begin{equation*}
\norm{\partial\varphi}_{\omega}^2 = \frac{1}{\cK}\norm{\sigma^*\log K_S}_{\omega_\Omega}^2 \equiv\frac{L_\Omega}{\cK} \geq \frac{rc_\Omega}{\cK} \;.
\end{equation*}
Table~\ref{data for BSDs} shows  that $rc_\Omega>n+1$  if $\Omega$ is irreducible and $\Omega\neq\Omega^{\mathrm{I}}_{1,n}=\UB^n$. Thus the assertion follows for an irreducible $\Omega$.

Suppose that $\Omega$ is a product of bounded symmetric domains $\Omega_1$ and $\Omega_2$, that is, $\Omega=\Omega_1\times\Omega_2$. Let $\omega_k$ ($k=1,2$) be the complete \KE metric of $\Omega_k$ with Ricci curvature $-\cK$. By the uniqueness of the \KE metric, we have 
\begin{equation*}
\omega=\pi_1^*\omega_1 + \pi_2^*\omega_2
\end{equation*}
as a K\"ahler form where $\pi_k:\Omega_1\times\Omega_2\to\Omega_k$ is the projection. Taking a potential function $\varphi_k$ of $\omega_k$ with constant gradient length, we have a potential function $\varphi=\pi_1^*\varphi_1 + \pi_2^*\varphi_2$ of $\omega$ satisfying
\begin{equation*}
\norm{\partial\varphi}_\omega^2 \equiv \norm{\partial\varphi_1}_{\omega_1}^2+\norm{\partial\varphi_2}_{\omega_2}^2 \;.
\end{equation*}
Proposition~\ref{prop:lower bound of length} implies that $\norm{\partial\varphi}_\omega^2$ is greater than $(n+1)/\cK$. This completes the proof.
\end{proof}



\subsection{Existence of a complete holomorphic vector field}\label{subsec:existence of a complete holomorphic vector field}
In \cite{LKH2021,Choi-Lee2021}, it was proved that there is a complete holomorphic vector field on a negatively curved complete \KE manifold admitting a global potential function with minimal constant gradient length.

\begin{theorem}[Theorem 3.2 in \cite{Choi-Lee2021}]\label{thm:existence}
Let $\omega$ be a complete \KE metric of $X^n$ with Ricci curvature $-K$. If  there is a  global potential $\varphi$ of $\omega$ satisfying
\begin{equation*}
\norm{\partial\varphi}_\omega^2 \equiv \frac{n+1}{\cK} \;,
\end{equation*}
then the $(1,0)$-vector field 
\begin{equation*}
V=\im e^{\frac{\cK\varphi}{n+1}}\grad(\varphi)
\end{equation*}
is a complete holomorphic vector field.
\end{theorem}

The completeness of $V$ means that the real tangent vector field given by $\RE V=V+\overline{V}$ is complete. Therefore the holomorphicity of $V$ implies that $\RE V$ generates infinitesimally an $1$-parameter family of automorphisms of $X$ which is nontrivial because $V$ is nowhere vanishing, namely, $\norm{V}_\omega = e^{\frac{\cK\varphi}{n+1}}\norm{\grad(\varphi)}_\omega = e^{\frac{\cK\varphi}{n+1}}\norm{\partial\varphi}_\omega$ is nowhere zero. In \cite{LKH2021,Choi-Lee2021}, we assumed that the Ricci curvature is $-(n+1)$ for $n=\dim X$ and $\varphi$ satisfies $d\dc\varphi=(n+1)\omega$ for a technical simplicity.  In what follows in this section, we will show Theorem~\ref{thm:existence} briefly.

\medskip

\noindent\textbf{The completeness of $V$}: Since $\omega$ is a complete metric, the $(1,0)$-vector field $W=\sqrt{-1}\grad(\varphi)$ with constant length $(n+1)/\cK$ is complete. Moreover the corresponding real vector field $\RE W$ is tangent to each level set of $\varphi$  since
\begin{equation*}
(\RE W)\varphi = \im\varphi^\alpha\varphi_\alpha - \im\varphi^{\bar\alpha}\varphi_{\bar\alpha} = 0 \;.
\end{equation*}
This means that an integral curve $\gamma:\RR\to X$ of $\RE W$ lies on a level subset $\set{\varphi=c}$. Since $V=e^{\frac{\cK\varphi}{n+1}}W = c'W$ on $\set{\varphi=c}$ for $c'=\cK c/(n+1)$, the curve $\tilde\gamma:\RR\to X$ given by $\tilde\gamma(t)=\gamma(c't)$ is an integral curve of $V$. This implies that $V$ is complete.

\medskip

\noindent\textbf{The holomorphicity of $V$}:
Let us written the vector field $V$ by
\begin{equation*}
V
=
\im V^\alpha\pd{}{z^\alpha}
\quad\text{where}\quad
V^\alpha
=
e^{\frac{\cK\varphi}{n+1}}\varphi^\alpha \;.
\end{equation*}
In order to prove that $V$ is holomorphic, we will see that $\nabla'' V$ is vanishing where $\nabla''$ is the $(0,1)$-part of the K\"ahler connection $\nabla$, so coincides with $\bar\partial$ on $T^{(1,0)}X$. The tensor field $\nabla'' V$ is given by
\begin{equation*}
\nabla'' V
=
\im\ind{V}{}{\alpha}{;\bar\beta}\pd{}{z^\alpha}\otimes dz^{\bar\beta}
\end{equation*}
where 
\begin{equation*}
	\ind{V}{}{\alpha}{;\bar\beta}
	=
	\frac{\cK}{n+1}e^{\frac{\cK\varphi}{n+1}}
	\varphi_{\bar\beta}\varphi^\alpha
	+
	e^{\frac{\cK\varphi}{n+1}}\ind{\varphi}{}{\alpha}{;\bar\beta} 
	=
	e^{\frac{\cK\varphi}{n+1}}\paren{
		\frac{\cK}{n+1}\varphi_{\bar\beta}\varphi^\alpha
		+\ind{\varphi}{}{\alpha}{;\bar\beta}
		}
	\;.
\end{equation*}
A straightforward computation gives that
\begin{multline*}
	\norm{\nabla'' V}_\omega^2
	\\
	=
	e^{\frac{2\cK\varphi}{n+1}}\paren{
	\paren{\frac{\cK}{n+1}}^2(\varphi_\alpha \varphi^\alpha)^2
	+
	\frac{\cK}{n+1} \varphi_{\bar\beta}
	\varphi^\alpha\ind{\varphi}{\alpha}{;\bar\beta}{}
	+
	\frac{\cK}{n+1}
	\ind{\varphi}{}{\alpha}{;\bar\beta}\varphi^{\bar\beta} \varphi_\alpha
	+
	\varphi^{\alpha;\beta}\varphi_{\alpha;\beta}
	}
		\\
	=
	e^{\frac{2\cK\varphi}{n+1}}\paren{
	\paren{\frac{\cK}{n+1}}^2\norm{\partial\varphi}^4_\omega
	-2\frac{\cK}{n+1} \norm{\partial\varphi}^2_\omega
	+
	1}
	\\
	=e^{\frac{2\cK\varphi}{n+1}}\paren{\frac{\cK}{n+1} \norm{\partial\varphi}^2_\omega-1}^2
	\equiv 0
	\;.
\end{multline*}
Here, we used the identity in \eqref{E:Key_equation}. This implies that $V$ is holomorphic.


\section{Proofs of main results}\label{sec:proof of main results}
In this section, we will prove Theorem~\ref{thm:main1} and Theorem~\ref{thm:main2}. Throughout this section, $X^n$ is a complex manifold with a complete \KE metric $\omega$ of Ricci curvature $-\cK$ and $\Gamma$ is a discrete, torsion-free, cocompact  subgroup of $\Aut(X)$.

\subsection{Proof of Theorem~\ref{thm:main1}}\label{subsec:proof of thm1}
Suppose that $X$ is simply connected and there is a global potential function $\varphi$ of $\omega$ with
\begin{equation}\label{eqn:minimal gradient}
\norm{\partial\varphi}_\omega^2\equiv \frac{n+1}{\cK} \;.
\end{equation}
By Lemma~\ref{lem:gradient length of BSDs}, it is sufficient to show that $X$ is biholomorphic to a bounded symmetric domain, i.e.  a hermitian symmetric space of noncompact type.

\medskip

Since $(X,\omega)$ is Ricci negative and $\Gamma$ acts on $X$ as an isometric transformation group, the quotient $\Gamma\backslash X$ has a negative anti-canonical class so $c_1(\Gamma\backslash X)<0$. By the Nadel-Frankel theorem (Theorem 0.1 in \cite{Frankel1995}), there is a finite covering $X'\to\Gamma\backslash X$ such that $X'$ is holomorphically factorized by
\begin{equation*}
X' = X'_1\times X'_2
\end{equation*}
where $X'_1$ is locally symmetric and $X'_2$ is locally rigid (the universal covering of $X_2'$ has a discrete automorphism group). Then we have the factorization
\begin{equation*}
X=X_1\times X_2
\end{equation*}
where $X_k$ is the universal covering of $X'_k$; therefore
\begin{enumerate}
\item $X_1$ is a hermitian symmetric space of noncompact type;
\item $\Aut(X_2)$ is discrete.
\end{enumerate}
We will show that $X_2$ is a trivial factor so that $X=X_1$.

\medskip

Suppose that $X_2$ is not trivial. Since $X_1$ is a hermitian symmetric space of noncompact type, so it admits a complete \KE metric $\omega_1$ with $\Ric(\omega_1)=-\cK\omega_1$. Moreover since $c_1(X'_2)<0$ by Corollary~4.5 in \cite{Frankel1995}, the covering $X_2$ also admits a complete \KE metric $\omega_2$ with $\Ric(\omega_2)=-\cK\omega_2$ by Yau~\cite{Yau1978}. Therefore $\omega$ should be the product metric of $\omega_1$ and $\omega_2$ because of the uniqueness of negatively curved complete \KE metric. 

\medskip

By the assumption as in \eqref{eqn:minimal gradient} and Theorem~\ref{thm:existence}, the $(1,0)$-vector field
\begin{equation*}
V=ie^{\frac{\cK\varphi}{n+1}}\grad(\varphi)
\end{equation*}
is a complete holomorphic vector field of  $X$. Let $\vV=V+\overline{V}$ be the corresponding real tangent vector field  and let $\set{\vV_t:t\in\RR}$ be its flow so that each $\vV_t$ belongs to the identity component of $\Aut(X)$ and $\Isom(X,\omega)$. By  the de Rham decomposition on a product space of simply connected Riemannian manifolds (Theorem~3.5 in Chapter VI of  \cite{KN-DG1}), $\vV_t$ can be split to isometries of $X_1$ and $X_2$; thus there is  $\vV_{k,t}\in\Isom(X_k,\omega_k)$ such that
\begin{equation*}
\vV_t(x_1,x_2)=(\vV_{1,t}(x_1),\vV_{2,t}(x_2))
\end{equation*}
for any $(x_1,x_2)\in X_1\times X_2$. That means that we can regard $\vV_{2,t}$ as an isometry of $(X_2,\omega)$. Since $\vV_t:X\to X$ is holomorphic, the restriction $\vV_{2,t}:X_2\to X_2$ is also holomorphic, so constitutes a holomorphic transformation group of $X_2$. Therefore $\vV_{2,t}$ is just the identity mapping of $X_2$ because $\Aut(X_2)$ is discrete. Now we can conclude that the infinitesimal generator $\vV$  should be tangent to each fiber $X_1$ of $X$ so orthogonal to each fiber $X_2$. This also holds for $\grad(\varphi)$. The identity $\ip{\grad(\varphi),\,\cdot\,}_\omega = \bar\partial\varphi(\,\cdot\,)$ says that $d\varphi(v)=\dc\varphi(v)=0$ for any vector  $v$ in the complexified tangent bundle $\CC TX_2$ of $X_2$.

Let $W$ be  a $(1,0)$-vector field tangent to $X_2$. The Lie bracket $[W,\overline{W}]$ is also tangent to $X_2$, so we have $\dc\varphi(W)=\dc\varphi(\overline{W})=\dc\varphi([W,\overline{W}])=0$. This means that the K\"ahler form $\omega$ annihilates the nontrivial subbundle $\CC TX_2$ since
\begin{multline*}
\omega(W,\overline{W}) = d\dc\varphi(W,\overline{W}) 
	= W\paren{\dc\varphi\overline{W}}-\overline{W}\paren{\dc\varphi(W)}-\dc\varphi([W,\overline{W}]) = 0 \;.
\end{multline*}
This is a contradiction, so $X_2$ is trivial.  \qed


\subsection{Proof of Theorem~\ref{thm:main2}}\label{subsec:proof of thm2}
Let $U$ be a localizing neighborhood $U$  of a sequence $\set{f_j}\subset\Gamma$, that is,  if $K\subset X$ is compact, then
\begin{equation*}
f_j(K)\subset U
\end{equation*}
for sufficiently large $j$. Suppose that there is a local potential function $\varphi:U\to\RR$ of $\omega$ satisfying
\begin{equation}\label{cond:uniform bound}
\norm{\partial\varphi}_\omega <C \quad\text{on $U$}
\end{equation}
for some constant $C$, and 
\begin{equation}\label{cond:boundary behavior in proof}
\lim_{j\to\infty} \norm{\partial\varphi}_\omega^2 (f_j(x)) = \frac{n+1}{\cK}
\end{equation}
for any $x\in X$.

\medskip

Let us fix a point $x_0\in X$ and consider a $\omega$-distance ball $B_R$ centered at $x_0$ with radius $R>0$. Then $B_R$ is relatively compact in $X$ so $f_j(B_R)\subset U$ eventually for $j$. Therefore we can consider a sequence $\set{\varphi_j}$ of functions on $B_R$ defined by
\begin{equation*}
\varphi_j = \varphi\circ f_j-(\varphi\circ f_j)(x_0) \;.
\end{equation*}
This is indeed a sequence of potentials of $\omega$ on $B_R$ since $f_j\in\Isom(X,\omega)$ implies
\begin{equation*}
d\dc\varphi_j = d\dc(f_j^*\varphi) = f_j^* d\dc\varphi = f_j^*\omega = \omega \;.
\end{equation*}
We will show that 
\begin{quote}
\textit{$\set{\varphi_j}$ admits a subsequence converging on $B_R$ in the local $C^\infty$-topology.}
\end{quote}
If it holds, using the compact exhaustion $X=\bigcup_{R_j\to\infty} B_{R_j}$ and the diagonal processing, we have a global potential function $\varphi_\infty$ of $\omega$ as a subsequential limit of $\set{\varphi_j}$ in the local $C^\infty$-topology of $X$. When we assume $\varphi_j\to\varphi_\infty$ passing to a subsequence, it follows
\begin{equation*}
\norm{\partial\varphi_\infty}_\omega (x) = \lim_{j\to\infty}\norm{\partial\varphi_j}_\omega (x)
= \lim_{j\to\infty}\norm{\partial\varphi}_\omega (f_j(x)) = \frac{n+1}{\cK}
\end{equation*}
for any $x\in X$ from $f_j\in\Isom(X,\omega)$ and the assumption of \eqref{cond:boundary behavior in proof}. Lifting $\varphi_\infty$ to the universal covering $\widetilde X$ of $X$ and applying Theorem~\ref{thm:main1}, we can see that $\widetilde X$ is biholomorphic to the unit ball.  In case of simply connected $U$, $X$ is also simply connected from the completeness of $\omega$ (see Lemma in pg. 256 in \cite{Wong1977}). Therefore it remains to show the assertion.

\medskip

Let us take $R'>R$ and consider $\varphi_j$ as a potential function  on $B_{R'}$ for sufficiently  large $j$ so that $f_j(B_{R'})\subset U$. Then we have
\begin{equation*}
\frac{1}{2}\norm{d\varphi_j}_\omega=\norm{\partial\varphi_j}_\omega = \norm{\partial (f^*\varphi)}_\omega = \norm{\partial\varphi}_\omega\circ f <C
\end{equation*}
uniformly on $B_{R'}$. Since $\varphi_j(x_0)=0$ for any $j$, we can conclude that $\set{\varphi_j}$ is uniformly bounded on $B_{R'}$.

\medskip

On order to show a subsequential convergence of $\set{\varphi_j}$ on $B_R\subset B_{R'}$ in the local $C^\infty$-topology of $B_R$, it suffices to prove that for each point $x\in B_R$, there is a neighborhood $U_x$ of $x$ such that $\set{\varphi_j}$ converges subsequentially in the local $C^\infty$-topology of $U_x$ since $B_R$ is relatively compact.

Take a sufficiently small, local coordinate neighborhood  $U_x\subset B_{R'}$ of a given $x\in B_R$ so that there is a local potential function $\varphi_x$ of $\omega$ on $U_x$ whose gradient length $\norm{d\varphi_x}_\omega$ is bounded. Then we have a sequence $\set{\varphi_j-\varphi_x}$ of uniformly bounded plurisubharmonic functions on $U_x$.  Now we can find $\psi_j:U_x\to\RR$ such that 
\begin{equation*}
\eta_j=\varphi_j-\varphi_x+\im\psi_j
\end{equation*}
is holomorphic solving $d\psi_j = 2\dc(\varphi_j-\varphi_x)$ on  $U_x$. When we normalize $\psi_j$ by $\psi_j(x)=0$, the sequence $\set{\psi_j}$ is also uniformly bounded on $U_x$ since $$\norm{d\psi_j}_\omega = \norm{2\dc(\varphi_j-\varphi_x)}_\omega = \norm{d(\varphi_j-\varphi_x)}_\omega\leq\norm{d\varphi_j}_\omega+\norm{d\varphi_x}_\omega.$$ The sequence $\set{\eta_j}$ of holomorphic functions on $U_x$ is uniformly bounded now; thus it admits a uniformly convergent subsequence on any compact subset of $U_x$. Simultaneously, $\set{\varphi_j}$ converges subsequentially in the local $C^\infty$-topology of $U_x$. This proves the assertion. \qed


\bibliographystyle{abbrv} 
\bibliography{CLS-quotient}

\end{document}